\documentclass[letterpaper, 10 pt, conference]{ieeeconf}  

\IEEEoverridecommandlockouts                              
\usepackage{xcolor}
\usepackage{empheq}
\usepackage{algorithm,algorithmic}

\usepackage{cite}
\usepackage{subcaption}
\usepackage{adjustbox}
\usepackage{textcomp, amssymb}
\usepackage[T1]{fontenc}
\usepackage{mathtools}
\usepackage{cleveref}

\RequirePackage[loading]{tracefnt}

\newtheorem{definition}{\bf Definition}

\newtheorem{proposition}{\bf Proposition}

\newcommand{\EE}{\mathbb{E}}

\newcommand{\RR}{\mathbb{R}}

\newcommand\by{\mathbf y}
\newcommand\bz{\mathbf z}

\newcommand\cS{\mathcal S}

\newcommand{\diag}{\mathrm{diag}}
\newlength\myindent
\title{\LARGE \bf
Policy Optimization for Linear-Quadratic \\
Zero-Sum Mean-Field Type Games
}

\usepackage{breqn}
\usepackage{amsmath}
\usepackage{bbm}

\let\inf\relax
\DeclareMathOperator*\inf{\vphantom{p}inf}

\author{Ren\'e Carmona, Kenza Hamidouche, Mathieu Lauri\`ere, and Zongjun Tan
\thanks{}
\thanks{R. Carmona, K. Hamidouche, M. Lauri\`ere and Z. Tan are with the Department of Operations Research and Financial Engineering, Princeton University, Princeton, NJ 08540, USA
        {\tt\small \{rcarmona, kenzah, lauriere, zongjun.tan\}@princeton.edu}}%
}

\begin{document}

\maketitle
\thispagestyle{empty}
\pagestyle{empty}

\begin{abstract}

In this paper, zero-sum mean-field type games (ZSMFTG) with linear dynamics and quadratic utility are studied under infinite-horizon discounted utility function. ZSMFTG are a class of games in which two decision makers whose utilities sum to zero, compete to influence a large population of agents. In particular, the case in which the transition and utility functions depend on the state, the action of the controllers, and the mean of the state and the actions, is investigated. The game is analyzed and explicit expressions for the Nash equilibrium strategies are derived. Moreover, two policy optimization methods that rely on policy gradient are proposed for both model-based and sample-based frameworks. In the first case, the gradients are computed exactly using the model whereas they are estimated using Monte-Carlo simulations in the second case. Numerical experiments show the convergence of the two players' controls as well as the utility function when the two algorithms are used in different scenarios. 

\end{abstract}

\section{Introduction}

Decision making in multi-agent systems has recently received an increasing interest from both theoretical and empirical viewpoints.
Multi-agent reinforcement learning (MARL) and stochastic games were shown to model well systems with a small number of agents. However, as the number of agents becomes large, analysing such systems becomes intractable due to the exponential growth of agent interactions and the prohibitive computational cost. To tackle this issue, mean-field approximations, borrowed from statistical physics, were considered to study the limit behaviour of systems in which the agents are indistinguishable and their decisions are influenced by the empirical distribution of the other agents. 

Mean-field games (MFGs) \cite{lasry2007mean,HuangMalhameCaines2006_MR2346927} and their variants mean-field type control (MFC) \cite{bensoussan2007representation} and mean-field type games (MFTG) \cite{barreiro2020discrete} consist of studying the global behaviour of systems composed of infinitely many agents which interact in a symmetric manner. In particular, the mean-field approximation captures all agent-to-agent interactions that, individually, have a negligible influence on the overall system's evolution.

An archetypal MFTG is mean-field zero-sum games. Two-player zero-sum games in their standard stochastic form, with no mean-field interactions, have been extensively studied in the literature. In this class of games, two decision makers compete to respectively maximize and minimize the same utility function. The large literature on this topic is motivated by many applications and by connections with robust control \cite{bacsar2008h}. Recently, generalizations to the case where the state dynamics is of MKV type have been introduced in continuous time over a finite horizon. Optimality conditions have been derived using the theory of backward stochastic differential equations (BSDEs)  in~\cite{xu2012zerosum}, using the dynamic programming principle and partial differential equations (PDEs) in~\cite{cosso2019zero} or using a weak formulation in~\cite{MR4104525}. All these works assume compactness of the action space, and hence are not applicable to a general linear-quadratic setting.

Although general stochastic problems with mean-field interactions can be studied from a theoretical perspective, explicit computation of the solution and numerical illustration of the Nash equilibrium are challenging. In standard optimal control, linear-quadratic (LQ) models, where the dynamics are linear and the utility is quadratic, usually have analytical or easily tractable solutions, which makes them very popular. These problems have also attracted much interest in the optimization and machine learning communities, since algorithms with proof of convergence can be developed, see e.g. \cite{fazel2018global} where the authors prove convergence of model-based and sample-based policy gradient methods for an LQ optimal control problem. Sample-based methods have also been used to solve LQ zero-sum games. In \cite{al2007model}, a discrete-time linear quadratic zero-sum game with infinite time horizon is studied and a Q-learning algorithm is proposed, which is proved to converge to the Nash equilibrium.  
In \cite{carmona2019linear}, the authors study mean-field control problems with a focus on linear-quadratic models in discrete time and propose a model-free policy gradient algorithm that is shown to converge to the optimal control. A model-free Q-learning algorithm is developed in \cite{carmona2019model} for MFC problems. The MFC problem is first cast as a Markov decision process (MDP) with deterministic transitions and then the convergence of the Q-learning algorithm is analysed. In \cite{zhang2019policy}, the authors study LQ zero-sum games and propose three projected nested-gradient methods that are shown to converge to the Nash equilibrium of the game. However, none of these works tackle mean-field interactions in a zero-sum setting.

In the present work, under an infinite-horizon
and discounted utility function, we investigate zero-sum mean-field type games (ZSMFTG) of linear-quadratic type, which, to the best of our knowledge, had not been the focus of any work before. In particular, we address the case in which the transition and utility functions do not only depend on the state and the action of the controllers, but also on the mean of the state and the actions. Moreover, the state is subject to a common noise. The structure of the problem and the infinite horizon regime allow us to identify the form of the equilibrium controls as linear combinations of the state and its mean conditioned on the common noise, both in the open-loop and the closed-loop settings. To learn the equilibrium, we extend the policy-gradient techniques developed in~\cite{carmona2019linear} for MFC, to the ZSMFTG framework. We design policy optimization methods in which the gradients are either computed exactly using the model or estimated using Monte-Carlo samples if the model is not fully known.

The rest of the paper is organized as follows. In Section~\ref{sec1:prob}, the zero-sum mean-field type game is formulated, preceded by a $N$-agent control problem which motivates this setting.  
Optimality conditions for the equilibrium are briefly discussed in Section~\ref{sec2:opt}, showing that the equilibrium controls of the two players are linear in the state and its mean. 
Model-based and model-free policy optimization methods are proposed in Section~\ref{sec4:algo}. In Section \ref{sec5:num}, we report numerical experiments to show the convergence of the controls and the utility function for different scenarios. Section~\ref{sec6:conc} concludes the paper. More details are provided in the long version of the paper~\cite{carmona2020lqzrmftglong}.

\section{Model and Problem Formulation}
In this section, we first present a zero-sum game in which two controllers compete to influence a population of agents interacting in a symmetric way, through the empirical distribution of their states and actions. We then present an asymptotic mean-field version of the game, in which the two controllers influence a state whose dynamics is of MKV type. 

\label{sec1:prob}
\subsection{$N$-agent problem}
Consider a system composed of a population $\{1,\dots,N\}$ with $N$ indistinguishable \emph{agents}. We investigate the case in which these agents have symmetric interactions and are influenced by two \emph{decision makers}, also called \emph{controllers} or \emph{players}, competing to optimize a criterion. In particular, we are interested in the linear-quadratic zero-sum case. Here, the state evolution of an agent $i\in\{1,\dots,N\}$ is given by
\begin{multline}
\label{eq:dyn-N-agent-general}
    x^i_{t+1} = A x^i_t + \bar{A} \bar{x}_t + B_{1} u^i_{1,t} + \bar{B}_{1} \bar{u}_{1,t} 
    \\
    + B_{2}u^i_{2,t} + \bar{B}_{2}\bar{u}_{2,t} 
    + \epsilon^i_{t+1} + \epsilon^0_{t+1},
\end{multline}
with initial condition $x^i_0 = \epsilon^i_0 + \epsilon^0_0$, where  $x^i_0$ is the initial state of agent $i$ to which we introduce randomness with $\epsilon^i_0$ and $ \epsilon^0_0$. At each time $t$, $x_t^i \in \RR^d$ corresponds to the state of the $i$-th agent in the population, and $u^i_{1,t} \in \RR^\ell$ and $u^i_{2,t} \in \RR^\ell$ are the controls prescribed to this agent respectively by the first and the second decision maker. The noise terms $\epsilon^0_{t+1}$ and $\epsilon^i_{t+1}$ are independent of each other and of $\epsilon^0_0$ and $ \epsilon^i_0$. Moreover, the noise terms $\epsilon^0_{t+1}$ for $t \ge 0$ are assumed to be identically distributed with mean $0$, and similarly for $\epsilon^i_{t+1}$ for $t \ge 0$. The interpretation of the noise terms is that $\epsilon^0_{t}$ is a common noise affecting the state of all the agents, whereas  $\epsilon^i_{t}$ is an indiosyncratic noise affecting only the state of the $i$-th agent. $A, \bar{A}, B_i, \bar{B}_i$ are fixed matrices with suitable dimensions.  Here, $\bar{x}_t=\frac{1}{N}\sum_{i=1}^{N} x_t^i$, is the sample average of the individual states, and similarly for $\mathrm{u}_1$ and $\mathrm{u}_2$: $\bar{u}_{j,t}=\frac{1}{N}\sum_{i=1}^{N} u_{j,t}^i$. The instantaneous utility is defined by
\begin{equation}
\label{eq:MKV_running_cost_c}
\begin{aligned}
    &c(x, \bar{x}, \mathrm{u}_{1}, \bar{\mathrm{u}}_{1}, \mathrm{u}_{2}, \bar{\mathrm{u}}_{2}) =
    (x-\bar{x})^{\top} Q (x-\bar{x}) + \bar{x}^{\top} (Q+\bar{Q}) \bar{x}
    \\ 
    &\qquad + (\mathrm{u}_{1}-\bar{\mathrm{u}}_{1})^{\top} R_{1} (\mathrm{u}_{1}-\bar{\mathrm{\bf \mathrm{u}}}_{1}) + \bar{\mathrm{u}}^\top_{1} (R_{1}+\bar{R}_{1}) \bar{\mathrm{u}}_{1} 
    \\
    &\qquad - (\mathrm{u}_{2} - \bar{\mathrm{u}}_{2})^\top R_{2} (\mathrm{u}_{2} - \bar{\mathrm{u}}_{2}) - \bar{\mathrm{u}}^\top_{2} (R_{2}+\bar{R}_{2}) \bar{\mathrm{u}}_{2}.
\end{aligned}
\end{equation} 
where $Q,\bar Q, R_i, \bar R_i$ are symmetric matrices of suitable sizes such that $R_i, R_i + \bar R_i$ for $i=1,2$ are positive definite.

The goal of each controller in this zero-sum problem is to minimize (resp. maximize) the $N$-agent utility functional 
\begin{equation*}
    J^N(\mathrm{\bf {\underline u}}_1, \mathrm{\bf {\underline u}}_2) = \EE\Big[\sum_{t=0}^{+\infty} \gamma^{t} \bar{c}^N({\underline x}_t, {\underline u}_{1,t}, {\underline u}_{2,t}) \Big],
\end{equation*}
where ${\underline x}_t = (x_t^1,\dots,x_t^N)$, and ${\bf {\underline u}}_i = ({\underline u}_{i,t})_t$ with ${\underline u}_{i,t} = (u^1_{i,t},\dots,u^N_{i,t})$ (we use a boldface to denote a function of time and an underline to denote a vector of size $N$), and $\bar c^N$ is the average utility, defined by
\begin{equation*}
    \bar{c}^N({\underline x}_t, {\underline u}_{1,t}, {\underline u}_{2,t}) 
    = \frac{1}{N} \sum_{i=1}^{N} c(x^i_t,\bar{x}_t, u^i_{1,t}, \bar{u}_{1,t}, u^i_{2,t}, \bar{u}_{2,t}).
\end{equation*}

 The minimax problem is defined as follows,
 \begin{equation}
        \inf_{\mathrm{\bf \underline{u}}_1} \sup_{\mathrm{\bf \underline{u}}_2} J^N(\mathrm{\bf \underline{u}}_1, \mathrm{\bf \underline{u}}_2).
\end{equation}

This problem is a generalization of the mean-field control setup, in which there is a single decision maker. It can also be viewed as a variant of Nash mean-field control setup studied in~\cite{bensoussan2018mean} or mean-field type games~\cite{djehiche2016mean} in which several mean-field decision makers compete in a general-sum game. 
An interesting special case is when each decision maker controls a different population. See~\cite[Remark 1]{carmona2020lqzrmftglong} for more details.

\subsection{Asymptotic mean-field problem}

Here, we consider the limit of the $N$-agent case. The dynamics is given by: $
    x_0 = \epsilon^0_0 + \epsilon^1_0,
$ and for $t \ge 0$, 
\begin{multline}
\label{eq:MKV-state_ZS}
    x_{t+1} = A x_t + \bar{A} \bar{x}_t + B_1 u_{1,t} + \bar{B}_1 \bar{u}_{1,t}
    \\
    + B_2 u_{2,t}+\bar{B}_2\bar{u}_{2,t}  
    + \epsilon^0_{t+1} + \epsilon^1_{t+1}.
\end{multline}

When considering the mean-field problem, we use the notation $\bar{x}_t = \EE[x_t | (\epsilon^0_{s})_{0 \le s \le t}]$ for the expectation of the state conditional on the realization of the common noise, and likewise for $\mathrm{\bf u}_1$ and $\mathrm{\bf u}_2$. Note that~\eqref{eq:MKV-state_ZS} is a dynamics of MKV type since it is influenced by its distribution and the distribution of the actions.
The utility function takes the form
\begin{equation}
\label{fo:MKV-discounted_utility_ZS}
    J(\mathrm{\bf u}_1, \mathrm{\bf u}_2) = \EE\Big[\sum_{t=0}^{+\infty} \gamma^{t} c_t \Big],
\end{equation}
where $\gamma \in [0,1]$ is a discount factor, and the instantaneous utility at time $t$ is defined as
\begin{align}
\label{eq:instantaneous-utility}
    c_t 
    & = c(x_t, \bar{x}_t, u_{1,t}, \bar{u}_{1,t}, u_{2,t}, \bar{u}_{2,t}),
\end{align}
where the function $c$ is as in the $N$-agent problem.  
The goal is to find a Nash equilibrium (NE), i.e., $( \mathrm{\bf u}^*_1, \mathrm{\bf u}^*_2)$ such that
\begin{align}
\label{eq:prob_for}
    J( \mathrm{\bf u}^*_1, \mathrm{\bf u}^*_2)
    =
    \inf_{\mathrm{\bf u}_1} \sup_{\mathrm{\bf u}_2} J(\mathrm{\bf u}_1, \mathrm{\bf u}_2).
\end{align}
Next, we study the existence of the NE and derive its closed-form expression for the formulated ZSMFTG.

\section{Optimality condition and gradient expression} 
\label{sec2:opt}

We now characterize the structure of a NE in terms of linear combinations of the state $x_t$ and conditional mean $\bar x_t$.

To alleviate the notation, let $\tilde A  = A + \bar {A}$, $\tilde Q  = Q + \bar {Q}$, $\tilde B_i  = B_i + \bar {B}_i$, $\tilde R_i  = R_i + \bar {R}_i$, $i=1,2$.  
Let us denote
\begin{align*} 
 	&\Gamma_i=(-1)^i\frac12 R_i^{-1}B_i^{\top},
 	\Xi_1 =  (-1)^i\frac12 R_i^{-1}\bigl[\bar{B}_i^{\top} - \bar{R}_i\tilde{R}_i^{-1}\tilde{B}_i^{\top}\bigr],
	\\
	&\Lambda_i =\Gamma_i+\Xi_i= (-1)^i \frac12 \tilde{R}_i^{-1}\tilde{B}_i^{\top}, \quad i=1,2.  
\end{align*}

To investigate the solution of (\ref{eq:prob_for}) and derive the closed-form expressions for the equilibrium controls in terms of the idiosyncratic and mean-field state processes, we introduce the following Riccati equations
\begin{equation}
	\label{eq:main_ARE_P_ZS}
		 \gamma [A^{\top} P + 2Q]\left[ A + \big(B_1 \Gamma_1 + B_2 \Gamma_2\big)  P \right] = P,
	\end{equation}
	and
	\begin{equation}
\label{eq:main_ARE_Pbar_ZS}
	\gamma\bigl[ \tilde A^{\top} \bar{P}+2 \tilde Q\bigr]\left[\tilde A +\big(\tilde B_1\Lambda_1 + \tilde B_2\Lambda_2\big)\bar{P} \right] =\bar{P}.
\end{equation}

Under suitable conditions and relying on a form of stochastic Pontryagin maximum principle (see \cite{bensoussan2007representation} for the zero-sum case without mean-field interactions and~\cite{CarmonaDelarue_book_I} for the case of mean-field interactions but without zero-sum structure), the ZSMFTG admits an open-loop Nash equilibrium, say $(\mathrm{\bf u}^*_1, \mathrm{\bf u}^*_2)$. These controls correspond to the open-loop saddle point and can be explicitly written in terms of the solutions $P, \bar P$ of the Riccati equations above as
	\begin{equation}
	\label{eq:MKV-opt-ctrl-formula}
		 u^*_{i,t} = \Gamma_i P(x_t - \bar x_t) + \Lambda_i \bar P \bar x_t, \textrm{ for } i=1,2.
 	\end{equation}

This relies on a form of Pontryagin maximum principle for mean-field dynamics. To keep the presentation concise, the proof is provided in the long version of the paper~\cite{carmona2020lqzrmftglong} (see Propositions~12 and~14, and Corollary~20 therein).

According to the above result, it is sufficient to look for  $( K^*_i, L^*_i), i=1,2$ such that
$
         u^*_{i,t} = (-1)^i K^*_i(x^{ \mathrm{\bf u}^*_1, \mathrm{\bf u}^*_2}_t - \bar{x}^{ \mathrm{\bf u}^*_1, \mathrm{\bf u}^*_2}_t) + (-1)^i L^*_i \bar{x}^{ \mathrm{\bf u}^*_1, \mathrm{\bf u}^*_2}_t,
    $
for $i=1,2$.

Optimizing over all possible open-loop controls is infeasible from a numerical perspective because it is a set of all stochastic processes which does not admit a simple representation. Hence, we focus on closed-loop Nash equilibrium with the above linear structure in the sequel (which allows us to do the optimization over a small number of parameters, namely the coefficients of the linear combination). In fact,  under suitable conditions, looking for closed-loop controls which are linear in $x$ and $\bar x$ leads to the same Nash equilibrium as open-loop controls, see~\cite[Section 6]{carmona2020lqzrmftglong}.

We henceforth replace problem~\eqref{eq:prob_for} by the following problem, for which optimality conditions are studied in~\cite[Section 5]{carmona2020lqzrmftglong}.  
Each player $i=1,2$ chooses parameter $\theta^*_i = ( K^*_i, L^*_i)$ 
such that
$$
    J(\mathrm{\bf u}^{ \theta^*_1}_1, \mathrm{\bf u}^{\theta^*_2}_2)
    =
    \inf_{\theta_1} \sup_{\theta_2} J(\mathrm{\bf u}^{\theta_1}_1, \mathrm{\bf u}^{\theta_2}_2),
$$
where for $\theta = (\theta_1,\theta_2)$,
$
    u^{\theta_1,\theta_2}_{i,t} = (-1)^i K_i(x^{\mathrm{\bf u}_1^{\theta_1},\mathrm{\bf u}_2^{\theta_2}}_t 
    - \bar{x}^{\mathrm{\bf u}_1^{\theta_1},\mathrm{\bf u}_2^{\theta_2}}_t) 
    + (-1)^i L_i \bar{x}^{\mathrm{\bf u}_1^{\theta_1},\mathrm{\bf u}_2^{\theta_2}}_t,
$ for $i=1,2$.

For simplicity, we introduce the following notation
$
    x^{\mathrm{\bf u}_1^{\theta_1},\mathrm{\bf u}_2^{\theta_2}}
    =
    x^{\theta_1,\theta_2},
$
and since we focus on linear controls, we redefine the utility as
$
    C(\theta_1,\theta_2)
    =
    J(\mathrm{\bf u}^{\theta_1}_1, \mathrm{\bf u}^{\theta_2}_2).
$
Moreover, we introduce $y^{K_1,K_2}_t = x^{\theta_1,\theta_2}_t - \bar{x}^{\theta_1,\theta_2}_t$ and $z^{L_1,L_2}_t = \bar{x}^{\theta_1,\theta_2}_t$, which is justified by the fact that the dynamics of $\by$ and $\bz$ depend respectively only on $(K_1,K_2)$ and $(L_1,L_2)$. 

Let $P^y_{K_1,K_2}$ be a solution to the linear equation
    \begin{align}
        &P^y_{K_1,K_2}=
        \label{eq:eq-Py_K1K2}
        Q + K_1^\top R_1 K_1 - K_2^\top R_2 K_2
        \\
        &+ \gamma (A - B_1 K_1 + B_2 K_2)^\top P^y_{K_1,K_2} (A - B_1 K_1 + B_2 K_2),
        \notag
    \end{align}
and let $P^z_{L_1,L_2}$ be a solution to the linear equation
    \begin{align}
        &P^z_{L_1,L_2}=
        \label{eq:eq-Pz_L1L2}
        \tilde Q  + L_1^\top \tilde R_1 L_1 - L_2^\top \tilde R_2 L_2
        \\
        \notag
        & + \gamma (\tilde A - \tilde B_1 L_1 + \tilde B_2 L_2)^\top P^z_{L_1,L_2} (\tilde A - \tilde B_1 L_1 + \tilde B_2 L_2).
    \end{align}

In order to guarantee that the above equations have solutions, we introduce the notion of stabilizing parameters. 
\begin{definition}
The set of stabilizing parameters is defined as follows:
\begin{align}
    \Theta =
    &\Big\{(K_1,L_1,K_2,L_2) \,:\,
    \gamma \|A - B_1 K_1 + B_2 K_2\|^2 < 1,
    \nonumber \\
    &\quad \gamma \| \tilde A - \tilde B_1 L_1 + \tilde B_2 L_2\|^2 < 1 
    \Big\}.
    \label{eq:Theta_stable_set}
\end{align}
\end{definition}

More details on this closed-loop information structure and the corresponding optimality conditions are provided in~\cite[Section 5]{carmona2020lqzrmftglong}.  

We now prove the following result, which provides an explicit expression for the gradient of the utility function with respect to the control parameters in terms of the solution to the equations~\eqref{eq:eq-Py_K1K2} and~\eqref{eq:eq-Pz_L1L2}.  
\begin{proposition}[Policy gradient expression]
\label{lem:PG_expression}
For any $\theta = (\theta_1,\theta_2) \in \Theta$, we have for $j=1,2$, the gradient
\begin{align}
\label{eq:PG_expression_Kj}
    \nabla_{K_j}C(\theta_1,\theta_2)
    & =
    2 E^{y,j}_{K_1,K_2} \Sigma^y_{K_1,K_2}
\end{align}
where
$    \begin{bmatrix}
        E^{y,1}_{K_1,K_2}\\
        E^{y,2}_{K_1,K_2}
    \end{bmatrix}
    = - \gamma  
        \begin{bmatrix} 
        B_1^\top  P^y_{K_1,K_2} A
        \\
       - B_2^\top  P^y_{K_1,K_2} A
        \end{bmatrix}
    +
    \mathbf{R} 
    \begin{bmatrix}
        K_1\\
        K_2
    \end{bmatrix}
$
with 
$$
    \mathbf{R} =
    \begin{bmatrix}
        R_1 + \gamma B_1^\top  P^y_{K_1,K_2} B_1 & - \gamma B_1^\top P^y_{K_1,K_2} B_2   
    \\
    - \gamma B_2^\top P^y_{K_1, K_2} B_1 &  - R_2 + \gamma B_2^\top  P^y_{K_1,K_2} B_2 
    \end{bmatrix}
$$
and
$
    \Sigma^y_{K_1,K_2}
    =\EE\left[ \sum_{t \ge 0} \gamma^t y^{K_1,K_2}_t (y^{K_1,K_2}_t)^\top \right].
$

Similarly, for $j=1,2$, 
$
    \nabla_{L_j}C(\theta_1,\theta_2)
    =
    2 E^{z,j}_{L_1,L_2} \Sigma^z_{L_1,L_2}
$ 
where
$
    \begin{bmatrix}
        E^{z,1}_{L_1,L_2}\\
        E^{z,2}_{L_1,L_2}
    \end{bmatrix}
    = - \gamma  
        \begin{bmatrix} 
        \tilde B_1^\top  P^z_{L_1,L_2} \tilde A
        \\
       - \tilde B_2^\top  P^z_{L_1,L_2} \tilde A
        \end{bmatrix}
    +
    \tilde {\mathbf{R}} 
    \begin{bmatrix}
        L_1\\
        L_2
    \end{bmatrix}
$
with 
$$
    \tilde {\mathbf{R}} =
    \begin{bmatrix}
        \tilde R_1 + \gamma \tilde B_1^\top  P^z_{L_1,L_2} \tilde B_1 & - \gamma \tilde B_1^\top P^z_{L_1,L_2} \tilde B_2   
    \\
    - \gamma \tilde B_2^\top P^z_{L_1,L_2} \tilde B_1 &  - \tilde R_2 + \gamma \tilde B_2^\top  P^z_{L_1,L_2} \tilde B_2 
    \end{bmatrix}
$$
and
$
    \Sigma^z_{L_1,L_2}
    =\EE\left[ \sum_{t \ge 0} \gamma^t z^{L_1,L_2}_t (z^{L_1,L_2}_t)^\top \right].
$
\end{proposition}
\begin{proof}
    We note that the utility can be split as 
    $
        C(\theta_1,\theta_2)
        =
        \EE_{\tilde y, \tilde z}\Big[
        C_{y}(K_1,K_2, \tilde y)
        +
        C_{z}(L_1,L_2, \tilde z)
        \Big],
    $
    where
    \begin{align*}
    &C_{y}(K_1,K_2, \tilde y) 
    =
    \EE \sum_{t \ge 0} \gamma^t \Big[
    (y^{L_1,L_2}_t)^{\top} Q y^{K_1,K_2}_t
    \\ 
    &\quad + \sum_{i=1}^2 (-1)^i(u_{i,t}-\bar{u}_{i,t})^{\top} R_{i} (u_{i,t}-\bar{u}_{i,t})  
    \,|\, y_0 = \tilde y
    \Big]
    \end{align*}
    and analogously for $C_{z}$. 
    Let us consider the first part. We note, using the above definition together with~\eqref{eq:eq-Py_K1K2} and the dynamics satisfied by $y^{K_1,K_2}_t = x^{\theta_1,\theta_2}_t - \bar{x}^{\theta_1,\theta_2}_t$, that
    \begin{align*}
    C_{y}(K_1,K_2, \tilde y) 
    &= \tilde y^{\top} P^y_{K_1,K_2} \tilde y
    + \frac{\gamma}{1-\gamma} \EE[(\epsilon^1_1)^{\top} P^y_{K_1,K_2} \epsilon^1_1],
    \end{align*}
    and thus 
    $
    \nabla_{\tilde y} C_y(K_1,K_2,\tilde y)
    = 2 P^y_{K_1,K_2} \tilde y.
    $
    Moreover,
    \begin{align*}
    &C_{y}(K_1,K_2, \tilde y) =
    \tilde y^\top (Q + K_1^\top R_1 K_1 - K_2^\top R_2 K_2) \tilde y
    \\
    &\quad + \gamma \EE \Big[
    C_{y}\Big(K_1,K_2, (A - B_1 K_1 + B_2 K_2) \tilde y\Big)
    \,|\, y_0 = \tilde y
    \Big].
    \end{align*}
    Using the two above equalities and the chain rule, we obtain 
    \begin{align*}
    &\nabla_{K_1} C_{y}(K_1,K_2, \tilde y) 
    \\
    &=
    2 R_1 K_1 \tilde y \tilde y^\top 
     - 2 \gamma B_1^\top  P^y_{K_1,K_2} (A - B_1 K_1 + B_2 K_2) \tilde y \tilde y^\top
    \\
    &\quad + \gamma \EE \Big[
    \nabla_{K_1} C_{y}\Big(K_1,K_2, \tilde y'\Big)_{\big| \tilde y' = (A - B_1 K_1 + B_2 K_2) \tilde y + \epsilon^1_1}
    \Big].
    \end{align*}
    Using recursion and the equation satisfied by $P^y_{K_1,K_2}$ yields
    \begin{align*}
    &\nabla_{K_1} C_{y}(K_1,K_2, \tilde y) 
    \\
    &
    =
    2 [R_1 K_1  - \gamma B_1^\top  P^y_{K_1,K_2} (A - B_1 K_1 + B_2 K_2)]
    \\
    &\quad 
    \Big(\tilde y \tilde y^\top 
    + \EE\Big[ \sum_{t \ge 1} \gamma^t y^{K_1,K_2}_t (y^{K_1,K_2}_t)^\top \Big] \Big) .
    \end{align*}
    With similar computations, we obtain $\nabla_{K_2} C_{y}$.   
    We proceed similarly for the gradients with respect to $L_1$ and $L_2$.
\end{proof}

\section{Proposed Algorithms}
\label{sec4:algo}

In this section, we propose policy-gradient algorithms to find the NE of the zero-sum MFTG. After introducing model-based methods, we explain how to extend them to sample-based algorithms in which the gradient is estimated using a simulator providing stochastic realizations of the utility.

\subsection{Model-based policy optimization} 
Let us assume that the model is known and both players can see the actions of one another at the end of each time step. 
To explain the intuition behind the iterative methods, we first express the optimal control of a player when the other player has a fixed control. 
 For some given $\theta_2=(K_2,L_2)$, the inner minimization problem for player $1$ becomes an LQR problem with instantaneous utility at time $t$:
\begin{align*}
    &(x_t-\bar{x}_t)^{\top} \mathbf{Q}_{K_2} (x_t-\bar{x}_t) + \bar{x}^{\top} \mathbf{\tilde Q}_{K_2} \bar{x}
    \\ 
    &\qquad + (u_{1,t}-\bar{u}_{1,t})^{\top} R_{1} (u_{1,t}-\bar{u}_{1,t}) + \bar{u}^\top_{1,t} (R_{1}+\bar{R}_{1}) \bar{u}_{1,t}, 
\end{align*}
when player $1$ uses control $u_1$, where $\mathbf{Q}_{K_2}=Q-K_2 R_2 K_2$ and $\mathbf{\tilde Q}_{L_2}= \tilde Q - L_2\tilde R_2 L_2$, and state dynamics given by:
\begin{multline}
    \nonumber
    x_{t+1} = \mathbf{A}_{K_2} x_t + \mathbf{\bar A}_{K_2,L_2} \bar{x}_t
    \\
    \nonumber 
    + B_1 u_{1,t} 
    + \bar{B}_1 \bar{u}_{1,t}
    + \epsilon^0_{t+1} + \epsilon^1_{t+1},
\end{multline}
where $\mathbf{A}_{K_2}=A+B_2K_2$ and $\mathbf{\bar A}_{K_2,L_2}=\bar A + \bar B_2L_2 + B_2(L_2 - K_2)$. 
Inspired by the results in~\cite{fazel2018global}, we propose to find the stationary point $\theta_1^*(\theta_2) = (K_1^*(K_2), L_1^*(L_2))$ of the inner problem. By setting
$\nabla_{\theta_1}C(\theta_1,\theta_2) = 0$ and by Proposition~\ref{lem:PG_expression},  
\begin{equation}
\label{eq:K1star_K2}
    K_1^*(K_2) = \gamma (R_1+\gamma B_1^{\top}P^{y}_{K_2}B_1)^{-1} B_1^{\top} P^y_{K_2}\left[A + B_2 K_2\right],
\end{equation}
where $P^y_{K_2} = P^y_{K_1^*(K_2),K_2}$ solves
  \begin{align*}
        &P^y_{K_2}
        =
        \tilde Q_{K_2} 
        + \gamma \tilde A_{K_2}^\top P^y_{K_2} \tilde A_{K_2}
        \\
        & 
        - \gamma^2 \tilde A_{K_2}^\top P^y_{K_2} B_1 (R_1+\gamma B_1^{\top}P^{y}_{K_2}B_1)^{-1} B_1^\top P^y_{K_2} \tilde A_{K_2},
    \end{align*}
where $\tilde{Q}_{K_2}=Q-K_2^{\top}R_2 K_2$ and $\tilde{A}_{K_2} = A+B_2K_2$. This equation is obtained by considering the equation~\eqref{eq:eq-Py_K1K2} for $P^y_{K_1,K_2}$ and replacing $K_1$ by the above expression~\eqref{eq:K1star_K2} for $K_1^*(K_2)$. One can similarly introduce $K_2^*(K_1)$, which is the optimal $K_2$ for a given $K_1$, and likewise for $L_1^*(L_2), L_2^*(L_1)$.

Based on this idea and inspired by the works of Fazel et al.~\cite{fazel2018global} and Zhang et al.~\cite{zhang2019policy}, we propose two iterative algorithms relying on policy-gradient methods, namely alternating-gradient and gradient-descent-ascent, to find the optimal values of $\theta_1$ and $\theta_2$. Starting from an initial guess of the control parameters, the players update either alternatively or simultaneously their parameters by following the gradients of the utility function. In the \emph{alternating-gradient} (AG) method, the players take turn in updating their parameters. Between two updates of $\theta_2$, $\theta_1$ is updated $T_1$ times. This procedure is summarized in Algorithm~\ref{algo:AG-ZSMFC}, which is based on nested loops. In the \emph{gradient-descent-ascent} (GDA) method, all the control parameters are updated synchronously at each iteration, as presented in Algorithm~\ref{algo:GDA-ZSMFC}. 

At each step of these methods, the gradients can be computed directly using the formulas provided in Proposition~\ref{lem:PG_expression}. For instance, in the inner loop of the AG method, based on~\eqref{eq:PG_expression_Kj}, parameter $K_1$ can be updated by:
\begin{align*}
    K_1^{t_1+1,t_2} 
    &= K_1^{t_1,t_2}-\eta_1 \nabla_{K_1} C(\theta_1^{t_1,t_2},\theta_2^{t_2-1})
    \\ 
    \nonumber 
    &=K^t_1-2\eta_1\Big[(R_1+B_1^{\top}P^{y}_{K_1^{t_1,t_2},K_2^{t_2-1}}B_1)K_1^{t_1,t_2}
    \\
    &\qquad\qquad -B^{\top}_1P^{y}_{K_1^{t_1,t_2},K_2^{t_2-1}}\tilde{A}_{K_2^{t_2-1}}\Big]\Sigma_{K_1^{t_1,t_2},K_2^{t_2-1}}.  
\end{align*}
Then, in the outer loop, one can compute $\nabla_{K_2} C$ at the point $(\theta_1^{T_1,t_2},\theta_2^{t_2-1})$ using again~\eqref{eq:PG_expression_Kj}.

In order to have a benchmark, one can compute the equilibrium $(\theta_1^*,\theta_2^*)$ by solving the Riccati equations~\eqref{eq:main_ARE_P_ZS}--\eqref{eq:main_ARE_Pbar_ZS} and then using the expression~\eqref{eq:MKV-opt-ctrl-formula}. Alternatively, the Nash equilibrium can be computed by finding $K_2$ such that $\nabla_{K_2} C_y(K_1^*(k_2),K_2)_{\big|k_2 = K_2} = 0$. The left-hand side has an explicit expression obtained by combining~\eqref{eq:PG_expression_Kj} and~\eqref{eq:K1star_K2}.

\begin{algorithm}
\caption{Alternating-Gradient method}
\label{algo:AG-ZSMFC} 
\begin{algorithmic} 
    \REQUIRE Number of inner and outer iterations $T_1, T_2$; initial guess $\theta_1^0, \theta_2^0$; learning rates $\eta_1,\eta_2$
    \ENSURE $(K^*_1,K^*_2)$ 
        \STATE $\theta_1^{0,1} \leftarrow \theta_1^0$\;
        \FOR{$t_2 = 1, 2, \dots, T_2$ }  
			\FOR{$t_1 = 1, 2, \dots, T_1$ }  
			\STATE  
			$\theta_1^{t_1,t_2} \leftarrow \theta_1^{t_1-1,t_2} - \eta_1 \nabla_{\theta_1} C(\theta_1^{t_1-1,t_2},\theta_2^{t_2-1})$\; 
	     	\ENDFOR
			\STATE $\theta_2^{t_2} \leftarrow \theta_2^{t_2-1} + \eta_2 \nabla_{\theta_2} C(\theta_1^{T_1,t_2},\theta_2^{t_2-1})$\; 
		\ENDFOR
		\RETURN $(\theta_1^{T_1,T_2},\theta_2^{T_2})$
\end{algorithmic}
\end{algorithm}   

\begin{algorithm}
\caption{Gradient-Descent-Ascent method}
\label{algo:GDA-ZSMFC} 
\begin{algorithmic} 
    \REQUIRE Number of iterations $T$; initial guess $\theta_1^{0}, \theta_2^0$; learning rates $\eta_1,\eta_2$
    \ENSURE $(K^*_1,K^*_2)$ 
        \FOR{$t = 1, 2, \dots, T$ } 
			\STATE  
			$\theta_1^{t} \leftarrow \theta_1^{t-1} - \eta_1 \nabla_{\theta_1} C(\theta_1^{t-1},\theta_2^{t-1})
			$\; 
			\STATE $\theta_2^{t} \leftarrow \theta_2^{t-1} + \eta_2 \nabla_{\theta_2} C(\theta_1^{t-1},\theta_2^{t-1})$\; 
	     	\ENDFOR
		\RETURN $(\theta_1^{T},\theta_2^{T})$
\end{algorithmic}
\end{algorithm}

\subsection{Sample-based policy optimization}

The aforementioned methods use explicit expressions for the gradients, which rely on the knowledge of the model.  
However, in many situations these coefficients are not known. Instead, let us assume that we have access to the following  
simulator, called \emph{MKV simulator} and denoted by $\cS^{\mathcal{T}}_{MKV}$: given a control parameter $\theta = (\theta_1, \theta_2) = (K_1, L_1, K_2, L_2)$, $\cS^{\mathcal{T}}_{MKV}(\theta)$ returns a sample of the mean-field utility (i.e., the quantity inside the expectation in equation~\eqref{fo:MKV-discounted_utility_ZS}) for the MKV dynamics~\eqref{eq:MKV-state_ZS} using the control $\theta$ and truncated at time horizon $\mathcal{T}$. This type of simulator is similar to the one introduced in~\cite{carmona2019linear} when there is a single controller. 

In other words, it returns a realization of the social utility $\sum_{t=0}^{\mathcal{T}-1} \gamma^t c_{t}$, where $c_t$ is the instantaneous mean-field utility at time $t$, see~\eqref{eq:instantaneous-utility}. This is used in Algorithm~\ref{algo:ZSMFC-MKVestim}, which provides a way to estimate the gradient of the utility with respect to the control parameters of the first player. One can estimate the gradient with respect to the control parameters of the second player in an analogous way. The estimation algorithm uses the simulator to obtain realizations of the (truncated) utility when using perturbed versions of the controls. In order to estimate the gradient of $C_y$, we use $2M$ perturbations $v_{1,1,i}, v_{1,2,i}$ which are i.i.d. with uniform distribution $\mu_{\mathbb{S}_\tau}$ over the sphere $\mathbb{S}_\tau$ of radius $\tau$. The first index corresponds to the player ($1$ or $2$), the second index corresponds to the part of the control being perturbed ($K$ or $L$) and the last index corresponds to the index of the perturbation (between $1$ and $M$). See e.g.~\cite{fazel2018global} for more details.  Notice that, although the simulator needs to know the model in order to sample the dynamics and compute the utilities, Algorithm~\ref{algo:ZSMFC-MKVestim} uses this simulator as a black-box (or an oracle), and hence uses only samples from the model and not the model itself.

\begin{algorithm}
	\caption{Sample-Based Gradient Estimation for Player~1}
	\label{algo:ZSMFC-MKVestim}
	\begin{algorithmic}
	\STATE {\bfseries Data:} {Parameter $\theta = (\theta_1, \theta_2) = (K_1,L_1,K_2,L_2)$; number of perturbations $M$; length $\mathcal{T}$; radius $\tau$}
	\STATE {\bfseries Result:} {An estimator for $\nabla_{\theta_1} C(\theta)$}
	 
		\FOR{$i = 1, 2, \dots, M$}
			\STATE Sample $v_{1,1,i}, v_{1,2,i}$ i.i.d. $\sim \mu_{\mathbb{S}_\tau}$\; 
			\STATE Set $\check\theta_{1,i} := (K_{1,i}, L_{1,i}) := (K_1 + v_{1,1,i}, L_1 + v_{1,2,i})$\;
			\STATE Set $\check\theta_i = (\check\theta_{1,i}, \theta_{2})$\; 
			\STATE Sample $\tilde{C}^i$ using MKV simulator  $\cS^{\mathcal{T}}_{MKV}(\check\theta_i)$\;
		\ENDFOR
		\STATE Set $\tilde{\nabla}_{K_1} C(\theta) =   \frac{d}{\tau^2} \frac{1}{M} \sum_{i=1}^M \tilde{C}^i v_{1,1,i},$ 
		\STATE and $\tilde{\nabla}_{L_1} C(\theta) = \frac{d}{\tau^2} \frac{1}{M} \sum_{i=1}^M \tilde{C}^i  v_{1,2,i}$ \;
		\STATE {\bfseries Return: }{$\tilde\nabla_{\theta_1} C(\theta)  := \diag\left(\tilde{\nabla}_{K_1} C(\theta), \tilde{\nabla}_{L_1} C(\theta)   \right)$}
\end{algorithmic}
\end{algorithm}

\section{Numerical Results}
\label{sec5:num}

In this section, we provide numerical results both for model-based and sample-based versions of the two methods presented in the previous section.  

\textbf{Setting.} The specification of the model used in the simulations is given in Table~\ref{tab:simulation_parameters_ZS}. This setting has been chosen so that it allows us to illustrate the convergence of the method when the equilibrium controls are not symmetric, i.e. $\theta_1 \neq \theta_2$. To be able to visualize the convergence of the controls, we focus on a one-dimensional example, that is, $d = \ell = 1$.

\textbf{Model-based results. }
The parameters used are given in Table~\ref{tab:simulation_parameters_ZS}. This choice of parameters is based on the values used for a single controller in~\cite{carmona2019linear} and numerical experiments.

Fig.~\ref{fig:1d-exact-surfaces} displays the trajectory of $(K_1,K_2) \mapsto C_y(K_1,K_2)$ and  $(L_1,L_2) \mapsto C_z(L_1,L_2)$ generated by the iterations of AG and DGA methods. Iterations are counted in the following way:  in AG at iteration $k$, $(\theta_1^k, \theta_2^k) = (\theta_1^{k \, \textrm{mod}\, T_1, \lceil k/T_1 \rceil}, \theta_2^{\lceil k/T_1 \rceil-1})$, while in DGA one step of for-loop corresponds to one iteration.  The utility at the starting point and the utility at the Nash equilibrium are respectively given by a black star and a red dot.
In the AG method, since $\theta_1$ is updated $T_1$ times between two updates of $\theta_2$, the trajectory moves faster in the  $\theta_1$-direction until it reaches an approximate best response against $\theta_2$, after which the trajectory moves towards the Nash equilibrium.  This is also confirmed by the convergence of the parameters $\theta = (K_1,L_1,K_2,L_2)$ in Fig.~\ref{fig:1d-exact-params}. 
The relative error on the utility is shown in Fig.~\ref{fig:1d-exact-utility}. The convergence is slower with AG because player $2$ updates her control only every $T_1$ iterations.

\begin{figure}
	\begin{subfigure}{0.5\columnwidth}
		\centering
		\includegraphics[width=1.05\columnwidth]{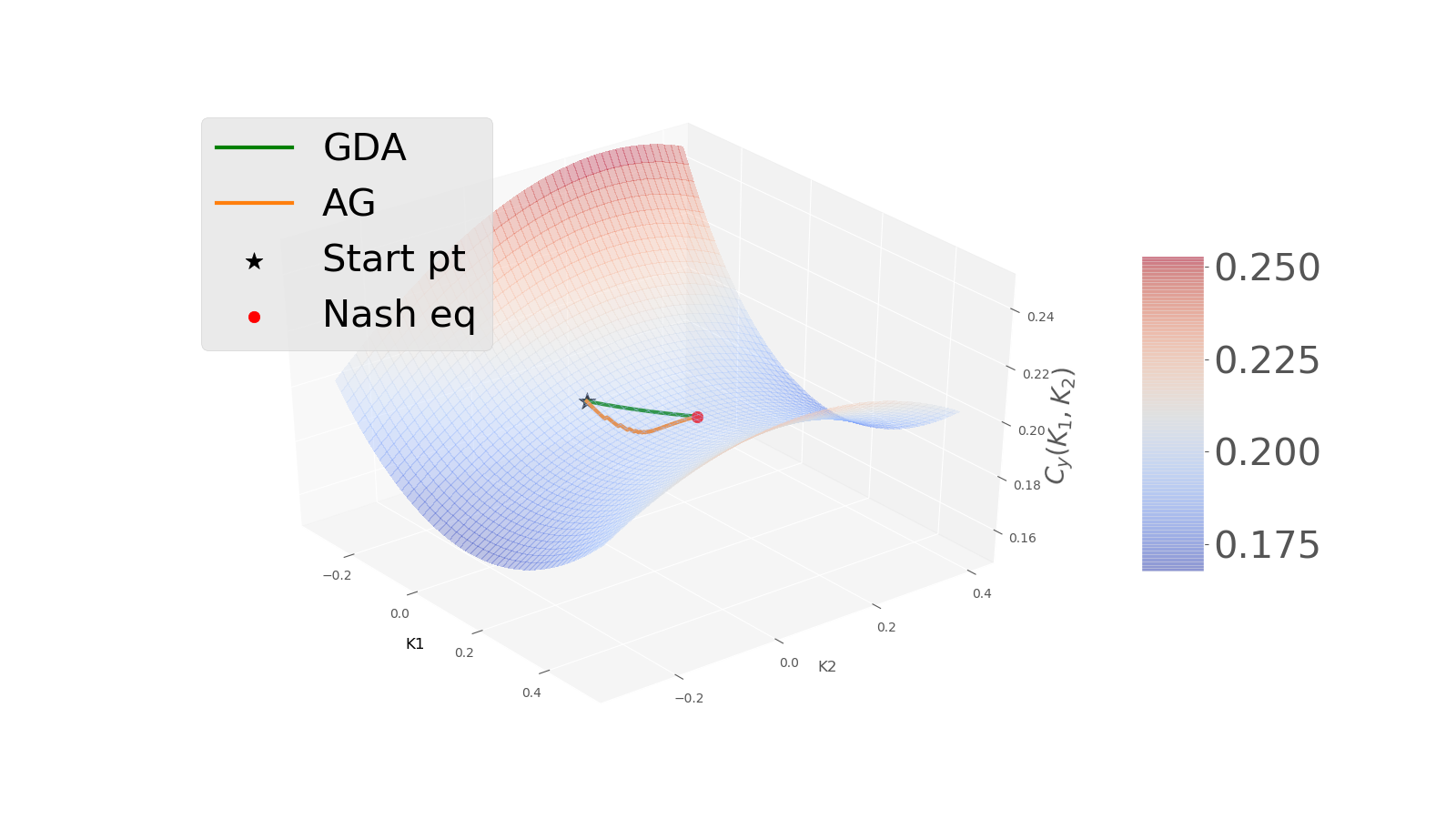}
		\caption{\,}
		\label{fig:1d-exact-surface-K}
	\end{subfigure}%
	\begin{subfigure}{0.5\columnwidth}
		\centering 
		\includegraphics[width=1.05\columnwidth]{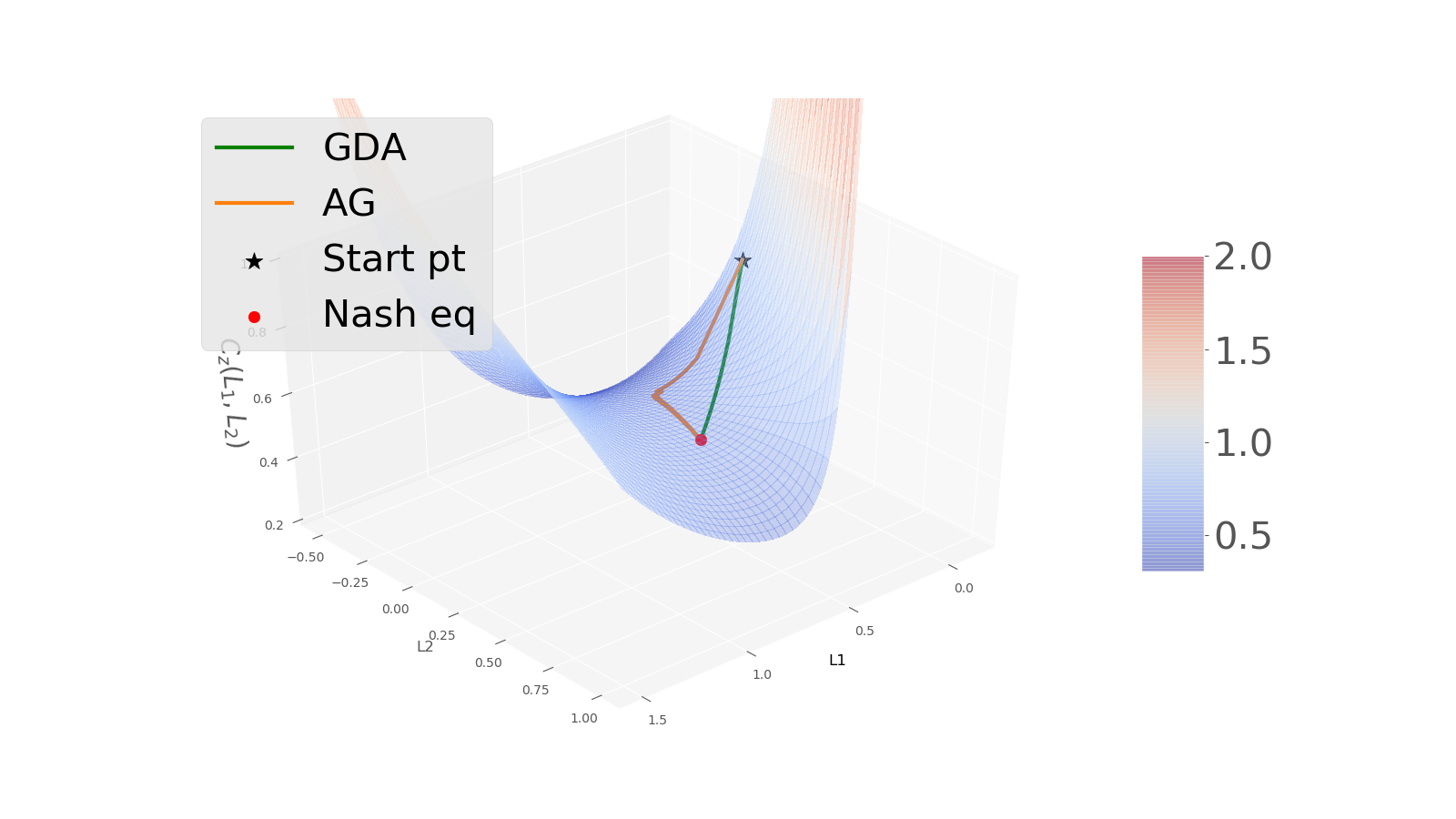}
		\caption{\,}
		\label{fig:1d-exact-surface-L}
	\end{subfigure}
	\caption{Model-based policy optimization: Convergence of each part of the utility. (a)~$C_y$ as a function of $(K_1,K_2)$. (b)~$C_z$ as a function of $(L_1,L_2)$.}
	\label{fig:1d-exact-surfaces}
\end{figure}

\begin{figure}
	\begin{subfigure}{.5\columnwidth}
		\centering 
		\includegraphics[width=1.05\columnwidth]{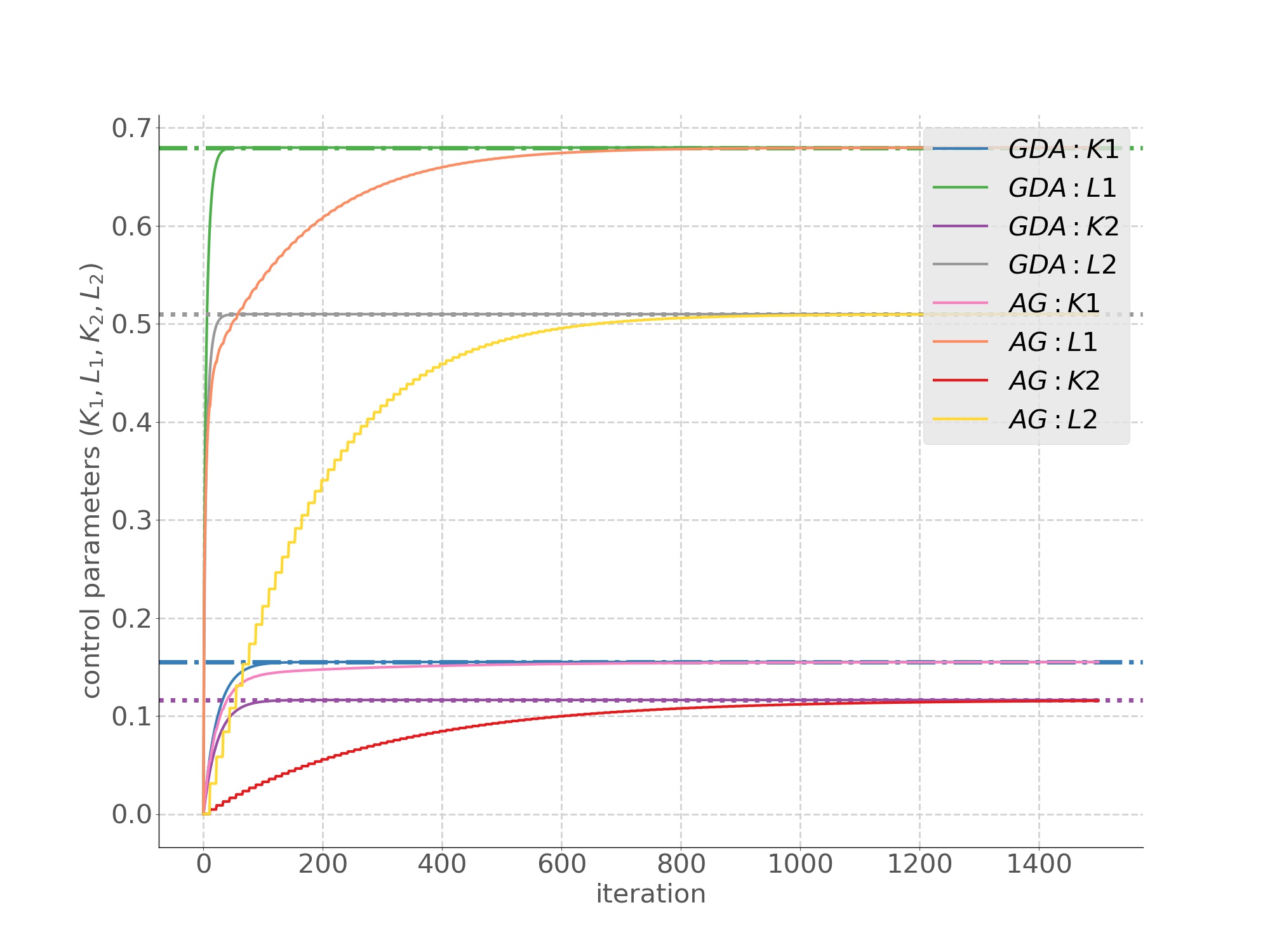}
		\caption{\,}
		\label{fig:1d-exact-params}
	\end{subfigure}%
	\begin{subfigure}{.5\columnwidth}
		\centering %
		\includegraphics[width=1.05\columnwidth]{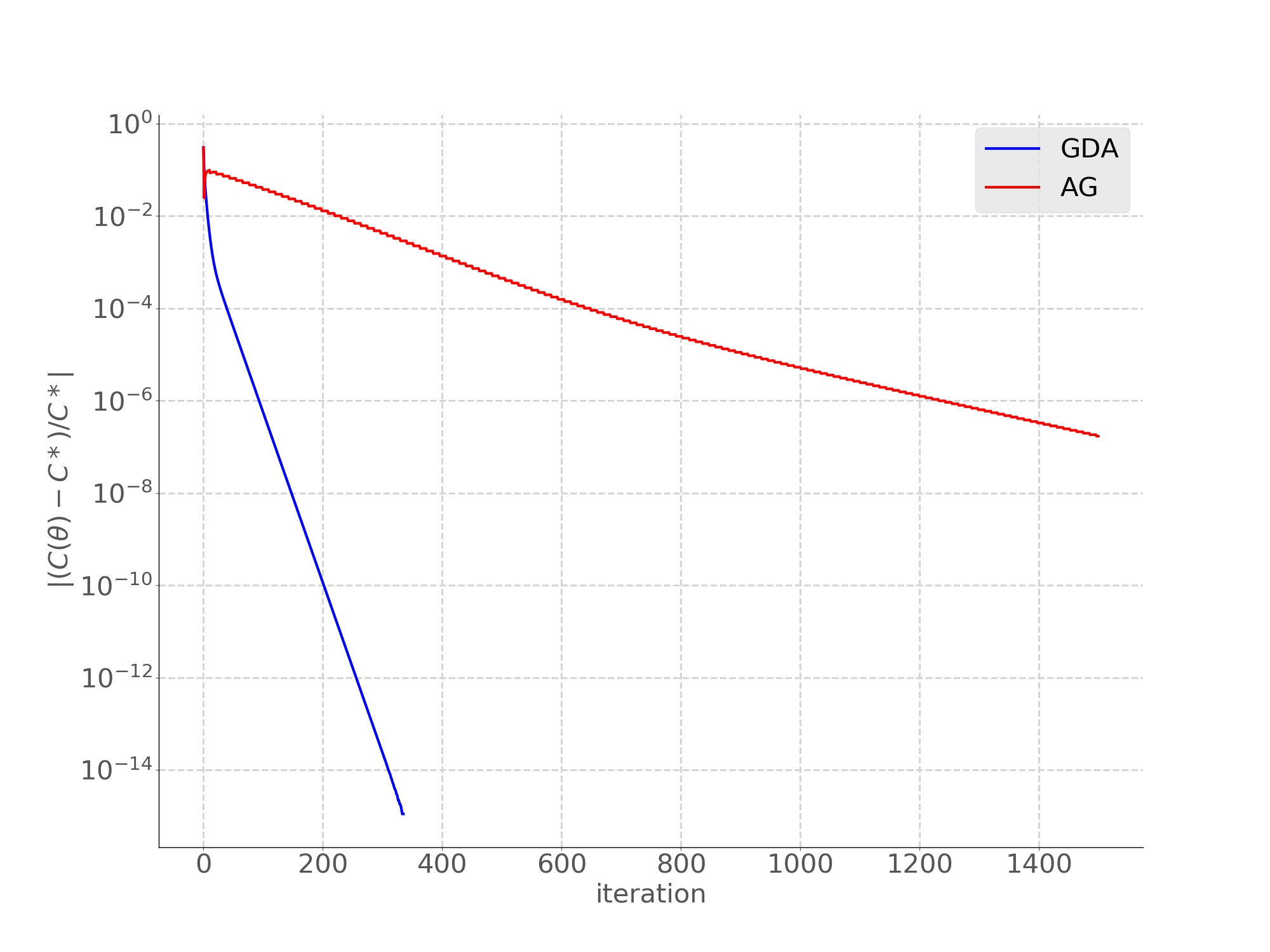}
		\caption{\,}
		\label{fig:1d-exact-utility}
	\end{subfigure}
	\caption{Model-based policy optimization: Convergence of the control parameters in~(a) and of the relative error on the utility in~(b).}
	\label{fig:1d-exact-params-utility}
\end{figure}

\textbf{Sample-based results. } The parameters, chosen based on the values in~\cite{carmona2019linear} as well as numerical experiments, are given in Table~\ref{tab:simulation_parameters_ZS}. The figures are obtained by averaging the results over 5 experiments, each based on a different realization of the randomness (initial points, dynamics and gradient estimation). 
Fig.~\ref{fig:1d-modelfree-surfaces} displays the trajectory of $(K_1,K_2) \mapsto C_y(K_1,K_2)$ and  $(L_1,L_2) \mapsto C_z(L_1,L_2)$ generated by the iterations of AG and DGA methods. The convergence of the parameters $\theta = (K_1,L_1,K_2,L_2)$ and the evolution of the relative error on the utility are shown in Fig.~\ref{fig:1d-modelfree-params} and~\ref{fig:1d-modelfree-utility}.

\begin{figure}
	\begin{subfigure}{0.5\columnwidth}
		\centering 
		\includegraphics[width=1.05\columnwidth]{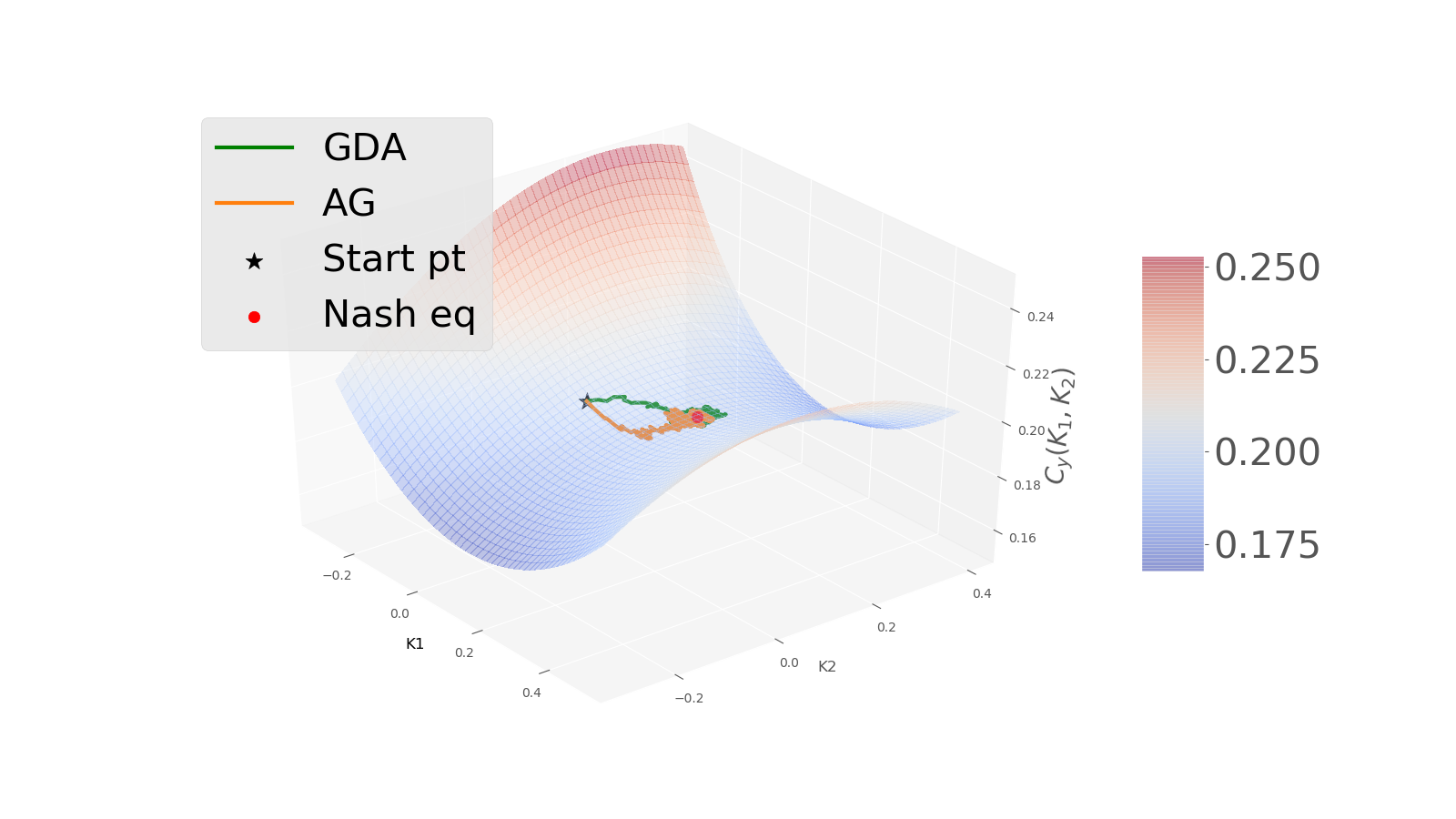}
		\caption{\,}
		\label{fig:1d-modelfree-surface-K}
	\end{subfigure}%
	\begin{subfigure}{0.5\columnwidth}
		\centering %
		\includegraphics[width=1.05\columnwidth]{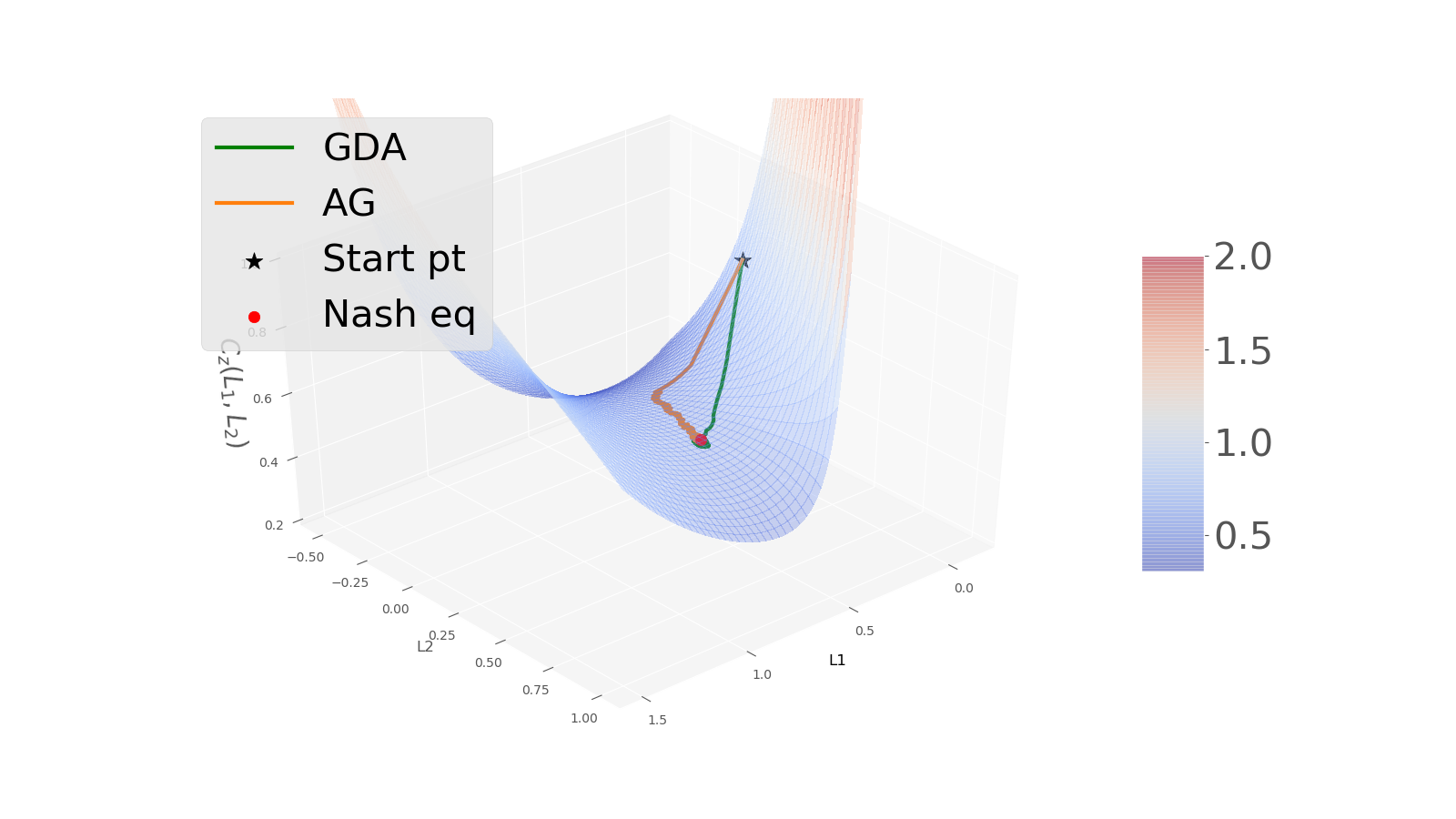}
		\caption{\,}
		\label{fig:1d-modelfree-surface-L}
	\end{subfigure}
	\caption{Sample-based policy optimization: Convergence of each part of the utility. (a)~$C_y$ as a function of $(K_1,K_2)$. (b)~$C_z$ as a function of $(L_1,L_2)$.}
	\label{fig:1d-modelfree-surfaces}
\end{figure}

\begin{figure}
	\begin{subfigure}{0.5\columnwidth}
		\centering 
		\includegraphics[width=1.05\columnwidth]{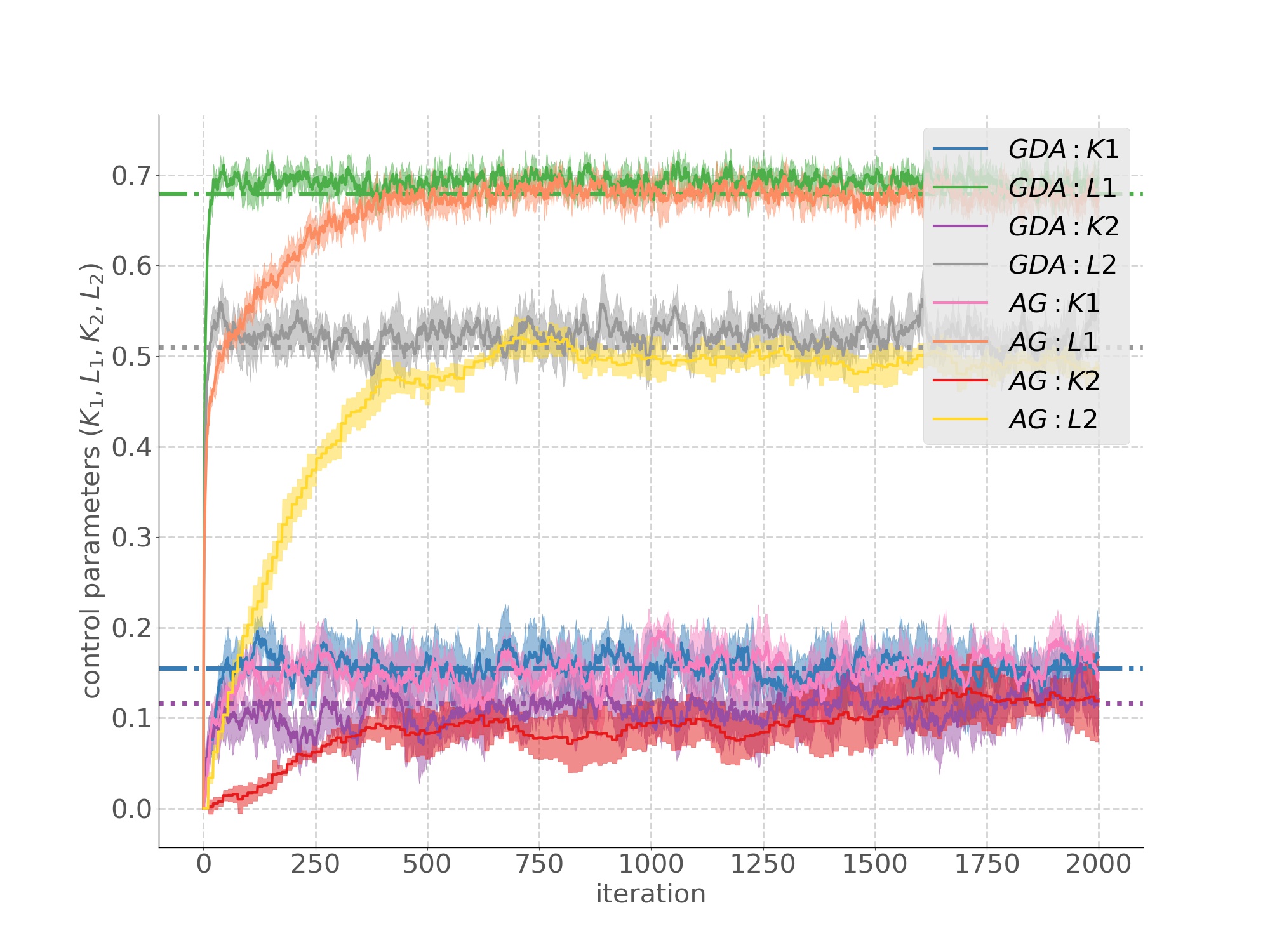}
		\caption{\,}
		\label{fig:1d-modelfree-params}
	\end{subfigure}%
	\begin{subfigure}{0.5\columnwidth}
		\centering %
		\includegraphics[width=1.05\columnwidth]{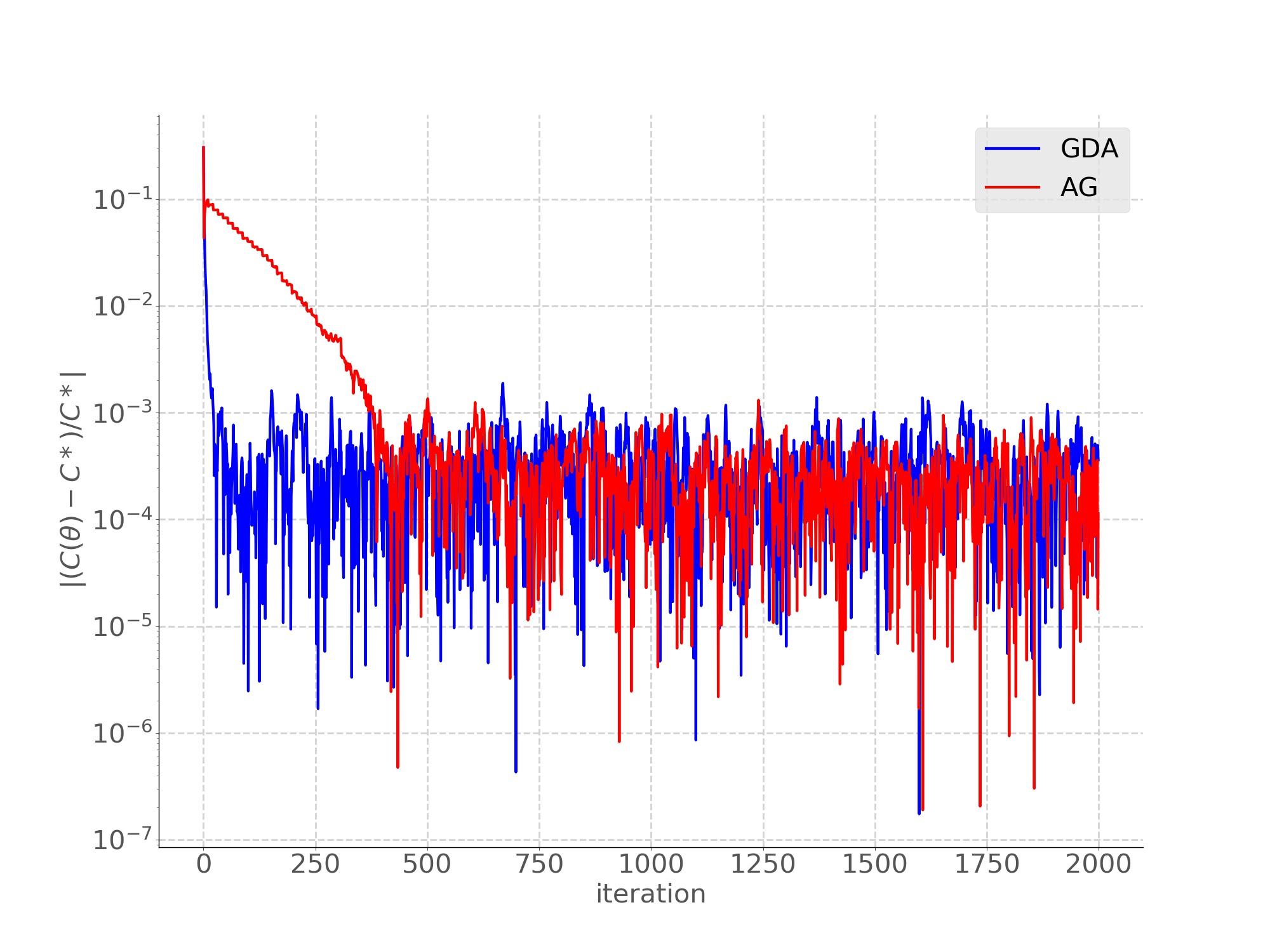}
		\caption{\,}
		\label{fig:1d-modelfree-utility}
	\end{subfigure}
	\caption{Sample-based policy optimization: Convergence of the control parameters in~(a) and of the relative error on the utility in~(b).}
	\label{fig:1d-modelfree-params-utility}
\end{figure}

\renewcommand{\arraystretch}{1.3}

\begin{table}[h!]
	
	\begin{center}
		\caption{Simulation parameters}       
		\begin{adjustbox}{width=0.95\columnwidth,center}
			\begin{tabular}{ccccccccc}
				
				\hline
				
				\multicolumn{9}{c}{Model parameters}\\
				
				\hline \hline
				
				 $A$ & $\overline{A}$ & $B_1=\overline{B}_1$ & $B_2=\overline{B}_2$ & $Q$ & $\overline{Q}$ & $R_1=\overline{R}_1$ & $R_2=\overline{R}_2$ & $\gamma$ \\
				
				\hline
				
				 0.4 & 0.4 & 0.4 & 0.3 & 0.4 & 0.4 & 0.4 & 0.4 & 0.9 \\
				
				\hline \hline
				
				
				\hline
				
				\multicolumn{9}{c}{Initial distribution and noise processes}\\
				
				\hline \hline
				
				 & \multicolumn{2}{c}{$\epsilon_0^0$} &  \multicolumn{2}{c}{$\epsilon^1_0$} &  \multicolumn{2}{c}{$\epsilon^0_t$} &  \multicolumn{2}{c}{$\epsilon^1_t$}  
				 \\
				
				\hline
				
				 & \multicolumn{2}{c}{$\mathcal{U}([-1, 1])$} & \multicolumn{2}{c}{$\mathcal{U}([-1, 1])$} & \multicolumn{2}{c}{$\mathcal{N}(0, 0.01)$} & \multicolumn{2}{c}{$\mathcal{N}(0, 0.01)$} 
				 \\
				
				\hline \hline
				
				
				\hline
				
				\multicolumn{9}{c}{AG and DGA methods parameters}\\
				
				\hline \hline
				
				$T_1$ & $T_2$ & $T$  & $\eta_1$ & $\eta_2$ & $K_1^0$ & $L_1^0$ & $K_2^0$ & $L_2^0$\\
				
				\hline
				
				10 & 200 & 2000 & 0.1 & 0.1 & 0.0 & 0.0 & 0.0 & 0.0 \\
				
				\hline \hline
				
				
				\hline
				
				\multicolumn{9}{c}{Gradient estimation algorithm parameters}\\
				
				\hline \hline
				\multicolumn{3}{c}{$\mathcal{T}$} &
				\multicolumn{3}{c}{$M$} &
				\multicolumn{3}{c}{$\tau$} 
				 \\
				
				\hline
				\multicolumn{3}{c}{50} &
				\multicolumn{3}{c}{10000} &
				\multicolumn{3}{c}{0.1} 
				\\
				
				\hline \hline

			\end{tabular}
		\end{adjustbox}

		\label{tab:simulation_parameters_ZS}
	\end{center}
	
\end{table}

\section{Conclusion}
\label{sec6:conc}
We have studied zero-sum mean-field type games with linear quadratic model under infinite-horizon discounted  utility  function. We have identified the  closed-form expression  of the  Nash equilibrium controls as linear combinations  of  the  state and  its  mean. Moreover, we have proposed two policy optimization methods to learn the equilibrium. Numerical results have shown the convergence of the two methods in both model-based and sample-based settings.

\section{Acknowledgments}

The research of M. Lauri\`ere is supported by NSF grant DMS--1716673 and ARO grant W911NF--17--1--0578.


{\footnotesize
\bibliographystyle{IEEEtran}
 \bibliography{references}}

\end{document}